\documentclass[12pt]{article}
\usepackage{mystyle}
\begin{document}
\title{The Rudin-Frol\'ik order and the Ultrapower Axiom}
\author{Gabriel Goldberg}
\maketitle
\begin{abstract}
We study the structure of the Rudin-Frol\'ik order on countably complete ultrafilters under the assumption that this order is directed. This assumption, called the Ultrapower Axiom, holds in all known canonical inner models. It turns out that assuming the Ultrapower Axiom, much more about the Rudin-Frol\'ik order can be determined. Our main theorem is that under the Ultrapower Axiom, a countably complete ultrafilter has at most finitely many predecessors in the Rudin-Frol\'ik order. In other words, any wellfounded ultrapower (of the universe) is the ultrapower of at most finitely many ultrapowers.
\end{abstract}
\section{Introduction}
This paper is about the structure of the class of countably complete ultrafilters under a simplifying assumption called the Ultrapower Axiom. We study how countably complete ultrafilters are built up as finite iterations of simple ultrafilters called {\it irreducible ultrafilters}. 

Irreducibility is defined in terms of an order on ultrafilters called the Rudin-Frol\'ik order which serves as a measure of how ultrapower embeddings can be factored as iterated ultrapowers. The Ultrapower Axiom itself is the assumption that the Rudin-Frol\'ik order is directed on countably complete ultrafilters. From this hypothesis, it turns out to be possible to derive many more properties of the Rudin-Frol\'ik order. 

Our main theorem answers the following question: given a wellfounded ultrapower \(M\) of the universe, how many distinct ultrapowers can \(M\) be an ultrapower of? This is essentially the question of whether a countably complete ultrafilter can have infinitely many predecessors in the Rudin-Frol\'ik order, which was raised in \cite{Kanamori} and \cite{Ketonen2}. Gitik \cite{Gitik} proved that the existence of such an ultrafilter is consistent with ZFC; this yields a wellfounded ultrapower of \(V\) that is the ultrapower of infinitely many distinct ultrapowers. The main result of this paper is that the Ultrapower Axiom implies the opposite answer: every ultrapower is the ultrapower of at most finitely many ultrapowers.

\section{Notation}
\begin{defn}
Suppose \(P\) is an inner model and \(U\) is an \(P\)-ultrafilter. We write \((M_U)^P\) to denote the ultrapower of \(P\) by \(U\) using functions in \(P\). We write \((j_U)^P\) to denote the ultrapower embedding from \(P\) to \((M_U)^P\) associated to \(U\). For any function \(f\in P\), we write \([f]^P_U\) to denote the point represented by \(f\) in \((M_U)^P\).
\end{defn}

We will often omit the parentheses in this notation, writing \(j_U^P\) and \(M_U^P\).

\begin{defn}
Suppose \(M\) and \(N\) are inner models. An elementary embedding \(i : M \to N\) is an {\it ultrapower embedding} if there is an \(M\)-ultrafilter \(U\) such that \(i = (j_U)^M\). An elementary embedding \(i : M \to N\) is an {\it internal ultrapower embedding} if there is an \(M\)-ultrafilter \(U\in M\) such that \(i = (j_U)^M\).

We say \(N\) is an {\it ultrapower} of \(M\) if there is an ultrapower embedding from \(M\) to \(N\). We say \(N\) is an {\it internal ultrapower} of \(M\) if there is an internal ultrapower embedding from \(M\) to \(N\). 
\end{defn}

Our definition of an ultrapower embedding reflects our focus on countably complete ultrafilters: {\it for example, an elementary embedding \(j : V\to M\) where \(M\) is illfounded is never an ultrapower embedding by our definition}. We note that there is a characterization of ultrapower embeddings that does not mention ultrafilters:
\begin{lma}
An elementary embedding \(j : M\to N\) is an ultrapower embedding if there is some \(a\in N\) such that \(N = H^N(j[M]\cup \{a\})\).
\end{lma}
\section{The Rudin-Frol\'ik order}
We now define the Rudin-Frol\'ik order and formulate the Ultrapower Axiom.

\begin{defn}
The {\it Rudin-Frol\'ik order} is defined on countably complete ultrafilters \(U\) and \(W\) by setting \(U\D W\) if there is an internal ultrapower embedding \(i : M_U\to M_W\) such that \(j_W = i\circ j_U\).
\end{defn}

By generalizing the definition of an internal ultrapower embedding, one can define the Rudin-Frol\'ik order on countably incomplete ultrafilters as well. A well-known fact about this more general order is that its restriction to ultrafilters on \(\omega\) forms a tree: the predecessors of an ultrafilter on \(\omega\) are linearly ordered by the Rudin-Frol\'ik order up to isomorphism. In particular, the Rudin-Frol\'ik order is not directed on ultrafilters on \(\omega\): otherwise it would be linear, contradicting a well-known result of Kunen \cite{Kunen} that states that even the Rudin-Keisler order on ultrafilters on \(\omega\) is not linear. 

Although the directedness of the Rudin-Frol\'ik order fails in essentially the simplest case, directedness {\it can} hold if one restricts to countably complete ultrafilters. This is the content of the Ultrapower Axiom:

\begin{ua}
The restriction of the Rudin-Frol\'ik order to countably complete ultrafilters is directed.
\end{ua}

\begin{defn}
Suppose \(M_0\), \(M_1,\) and \(N\) are transitive models of set theory. We write \((i_0,i_1) : (M_0,M_1)\to N\) to denote that \(i_0 : M_0\to N\) and \(i_1 : M_1\to N\) are elementary embeddings.
\end{defn}

\begin{defn}
Suppose \(j_0 : V\to M_0\) and \(j_1 : V\to M_1\) are ultrapower embeddings. A {\it comparison} of \((j_0,j_1)\) is a pair \((i_0,i_1) : (M_0,M_1)\to N\) of internal ultrapower embeddings such that \(i_0 \circ j_0 = i_1\circ j_0\).
\end{defn}

The following lemma, which is immediate given the definitions, partially explains the relationship between the Ultrapower Axiom and the comparison lemma of inner model theory.

\begin{lma}
The following are equivalent:
\begin{enumerate}[(1)]
\item The Ultrapower Axiom holds.
\item Every pair of ultrapower embeddings admits a comparison.
\end{enumerate}
\end{lma}
\section{A basic fact}
The following lemma highlights an important difference between the Rudin-Frol\'ik order and the Rudin-Keisler order.

\begin{prp}\label{RFUnique}
Suppose \(U\D W\). Then there is a unique internal ultrapower embedding \(k : M_U\to M_W\) such that \(j_W = k\circ j_U\).
\end{prp}

For the proof we use the following basic schema.

\begin{lma}\label{DefinableUnique}
Any two \(\Sigma_2\)-definable embeddings from \(V\) into a common inner model agree on the ordinals.
\begin{proof}[Sketch]
Suppose towards a contradiction that \(\alpha\) is the least ordinal such that there exist \(\Sigma_2\)-definable elementary embeddings from \(V\) into the same inner model that differ at \(\alpha\). Note that \(\alpha\) is definable without parameters. Thus if \(j,j' : V\to M\) are elementary embeddings, \(j(\alpha) = j'(\alpha)\). This contradicts the definition of \(\alpha\).
\end{proof}
\end{lma}

\begin{proof}[Proof of \cref{RFUnique}]
Suppose \(k,k' : M_U\to M_W\) are internal ultrapower embeddings with \(j_W = k\circ j_U = k' \circ j_U\). Then by \cref{DefinableUnique} applied in \(M_U\), \(k\restriction \text{Ord} = k'\restriction \text{Ord}\). Now \(M_U = H^{M_U}(j_U[V]\cup \text{Ord})\) and \(k\restriction (j_U[V]\cup \text{Ord}) = k'\restriction (j_U[V]\cup \text{Ord})\). Therefore \(k = k'\).
\end{proof}
\section{Canonical comparisons}
In this section we prove that under UA, any pair of countably complete ultrafilters \(U_0\) and \(U_1\) has a {\it least} upper bound in the Rudin-Frol\'ik order. By this we mean that there is a  countably complete ultrafilter \(W\) such that \(U_0,U_1\D W\) and for any \(W'\) such that \(U_0,U_1\D W'\), \(W\D W'\). (So \(W\) is the least upper bound of \(U_0,U_1\) among countably complete ultrafilters; we do not know whether it is least among arbitrary ultrafilters.) 

The notation is slightly less cumbersome if one works with ultrapower embeddings instead of ultrafilters, so this is how we will proceed.

\begin{defn}
Suppose \(j_0: V\to M_0\) and \(j_1 : V\to M_1\) are ultrapower embeddings. A comparison \((i_0,i_1) : (M_0,M_1)\to N\) of \((j_0,j_1)\) is {\it canonical} if for any other comparison \((i'_0,i'_1) : (M_0,M_1)\to N'\) there is an elementary embedding \(h : N \to N'\) such that \(i_0' = h \circ i_0\) and \(i'_1 = h\circ i_1\).
\end{defn}

We begin by proving the existence and uniqueness of canonical comparisons assuming UA. Uniqueness is actually provable in ZFC by an almost standard category theoretic argument. The only twist is that we use the following standard fact which (essentially) appears as Theorem 9.2 in \cite{Comfort}. We supply a proof in the countably complete case, which is significantly easier.

\begin{thm}\label{RKFact}
Suppose \(k : V\to N\) is an ultrapower embedding and \(e : N \to N\) is an elementary embedding with \(e\circ k = k\). Then \(e\) is the identity.
\begin{proof}
Let \(\alpha\) be the least ordinal such that \(N = H^N(k[V]\cup \{\alpha\})\). It suffices to show that \(e(\alpha) = \alpha\). Suppose towards a contradiction that \(e(\alpha) > \alpha\). Then since \(N = H^N(k[V]\cup \{\alpha\})\), there is some function \(f\) such that \(e(\alpha) = k(f)(\alpha) = e(k(f))(\alpha)\). Therefore \(N\) satisfies the statement that there is some \(\beta < e(\alpha)\) such that \(e(\alpha) = e(k(f))(\beta)\). By the elementarity of \(e : N \to N\), \(N\) satisfies that there is some \(\beta < \alpha\) such that \(\alpha = k(f)(\beta)\). But then \(N = H^N(k[V]\cup \{\beta\})\), contrary to the minimality of \(\alpha\).
\end{proof}
\end{thm}

\begin{lma}\label{CanonicalUnique}
Suppose \(j_0: V\to M_0\) and \(j_1 : V\to M_1\) are ultrapower embeddings. Then \((j_0,j_1)\) has at most one canonical comparison.
\begin{proof}
Suppose \((i_0,i_1) : (M_0,M_1)\to N\) and \((i_0',i_1') : (M_0,M_1)\to N'\) are canonical comparisons of \((j_0,j_1)\). We will show that \(N = N'\), \(i_0 = i_0'\), and \(i_1 = i_1'\).  

Fix \(h : N \to N'\) such that \(i_0' = h\circ i_0\) and \(i_1' = h\circ i_1\), and fix \(h' : N'\to N\) such that \(i_0 = h'\circ i_0'\) and \(i_1 = h'\circ i_1'\). Let \(k = i_0 \circ j_0 = i_1\circ j_1\), so \(k : V\to N\) is an ultrapower embedding. Let \(e = h'\circ h\). Then \(e : N \to N\) and \(e\circ k = k\). By \cref{RKFact}, \(e = \text{id}\). It follows that \(h : N \to N'\) is surjective, and hence \(N = N'\) and \(h = \text{id}\). Thus \(i_0 = h\circ i_0 = i'_0\) and \(i_1 = h\circ i_1 = i'_1\), as desired.
\end{proof}
\end{lma}

The existence of canonical comparisons will use UA. The plan is to show that under UA, comparisons with a certain easily obtainable minimality property are actually canonical comparisons.

\begin{defn}\label{MinimalityDef}
A pair \((i_0,i_1) : (M_0,M_1)\to N\) is {\it minimal} if \[N = H^N(i_0[M_0]\cup i_1[M_1])\]
\end{defn}

We do not assume that \(i_0\) and \(i_1\) are ultrapower embeddings in the definition of minimality. Therefore the following lemma has some content.

\begin{lma}\label{MinimalUltrapowers}
Suppose \(j_0 : V\to M_0\) and \(j_1 : V\to M_1\) are ultrapower embeddings and \((i_0,i_1) : (M_0,M_1)\to N\) is minimal and satisfies \(i_0\circ j_0 = i_1\circ j_1\). Then \(i_0\) and \(i_1\) are ultrapower embeddings.
\begin{proof}
We show that \(i_0\) is an ultrapower embedding. Fix \(a\in M_1\) such that \(M_1 = H^{M_1}(j_1[V]\cup \{a\})\). Then \[N = H^N(i_0[M_0]\cup i_1[M_1]) = H^N(i_0[M_0]\cup i_1\circ j_1[V]\cup \{i_1(a)\})= H^N(i_0[M_0]\cup \{i_1(a)\})\]
It follows that \(i_0 : M_0\to N\) is an ultrapower embedding (given by the ultrafilter derived from \(i_0\) using \(i_1(a)\)).

The fact that \(i_1\) is an ultrapower embedding follows by a similar argument.
\end{proof}
\end{lma}

The following fairly obvious lemma is often useful:

\begin{lma}\label{InternalClose}
Suppose \(i' : M\to N'\) is an internal embedding, \(i  : M \to N\) is an ultrapower embedding, and \(h :  N \to N'\) is an elementary embedding such that \(h\circ  i = i'\). Then \(i\) is an internal ultrapower embedding.
\begin{proof}
Since \(i\) is an ultrapower embedding, we may fix \(a\in N\) such that \( N = H^{ N}(i[M]\cup \{a\})\). Letting \(U\) be the ultrafilter derived from \(i'\) using \(h(a)\), it is not hard to show that \(i = (j_U)^M\). Since \(i'\) is an internal embedding, \(U\in M\), and hence \(i\) is an internal ultrapower embedding.
\end{proof}
\end{lma}

A hull argument now yields the existence of a minimal comparison of any two comparable ultrapowers.

\begin{lma}\label{Hull}
Suppose \(j_0 : V\to M_0\) and \(j_1 : V\to M_1\) are ultrapower embeddings and \((i_0',i_1') : (M_0,M_1)\to N'\) is a comparison of \((j_0,j_1)\). Then \((j_0,j_1)\) admits a unique minimal comparison \((i_0,i_1) : (M_0,M_1)\to N\) such that there is an elementary embedding \(h : N \to N'\) with \(i_0' = h\circ i_0\) and \(i_1' = h\circ i_1\).
\begin{proof}
Let \(H = H^{N'}(i_0'[M_0]\cup i_1'[M_1])\) and let \(h : N \to N'\) be the inverse of the transitive collapse of \(H\). Let \(i_0 = h^{-1}\circ i_0'\) and \(i_1 = h^{-1}\circ i_1'\). Easily \(i_0\circ j_0 = i_1\circ j_1\). By \cref{MinimalUltrapowers}, \(i_0\) and \(i_1\) are ultrapower embeddings. By \cref{InternalClose}, \(i_0\) and \(i_1\) are internal ultrapower embeddings. Since \(H = H^{N'}(i_0'[M_0]\cup i_1'[M_1])\), tracing through the isomorphism \(h : N \to H\) easily shows \(N = H^N(i_0[M_0]\cup i_1[M_1])\). Thus \((i_0,i_1)\) is a minimal comparison of \((j_0,j_1)\).

We finally show the uniqueness of \((i_0,i_1)\). Suppose  \((i^*_0,i^*_1) : (M_0,M_1)\to N^*\) is a minimal comparison of \((j_0,j_1)\) such that there is an elementary embedding \(h^* : N^* \to N'\) with \(i_0' = h^*\circ i^*_0\) and \(i_1' = h^*\circ i^*_1\). We must have \(h^*[N^*] = H^N(i_0'[M_0']\cup i_1'[M_1']) = H\), and therefore \(h^*\) is the inverse of the transitive collapse of \(H\), so that \(N^* = N\) and \(h^* = h\). It follows that \(i_0^* = h^{-1}\circ i_0' = i_0\) and \(i_1^* = h^{-1}\circ i_1' = i_1\).
\end{proof}
\end{lma}

\begin{lma}
Canonical comparisons are minimal.
\begin{proof}
Suppose \(j_0 : V\to M_0\) and \(j_1 : V\to M_1\) are ultrapower embeddings and \((i'_0,i'_1) : (M_0,M_1)\to N'\) is the canonical comparison of \((j_0,j_1)\). By \cref{Hull}, there is a minimal comparison \((i_0,i_1) : (M_0,M_1)\to N\) and an elementary embedding \(h : N\to N'\) such that \(i_0' = h\circ i_0\) and \(i_1' =h\circ i_1\). It follows immediately that \((i_0,i_1)\) is a canonical comparison. Thus by \cref{CanonicalUnique}, \(N = N'\), \(i_0 = i_0'\), and \(i_1 = i_1'\). It follows that \((i_0',i_1')\) is minimal, as desired.
\end{proof}
\end{lma}

Combining this with the following lemma, one can strengthen the definition of a canonical comparison to assert that the embedding \(h : N \to N'\) is unique.

\begin{lma}\label{MinimalEmb}
Suppose \((i_0,i_1) : (M_0,M_1)\to N\) is minimal and \((i_0',i_1') : (M_0,M_1)\to N'\). Then there is at most one embedding \(h : N \to N'\) such that \(i_0' = h\circ i_0\) and \(i_1'= h\circ i_1\).
\begin{proof}
The requirements  \(i_0' = h\circ i_0\) and \(i_1'= h\circ i_1\) determine \(h \restriction i_0[M_0]\) and \(h\restriction i_1[M_1]\). Since \(N = H^N(i_0[M_0]\cup i_1[M_1])\), this determines \(h\) on all of \(N\).
\end{proof}
\end{lma}

We also have the following fact that yields the exact relationship between minimality and canonicity:

\begin{prp}\label{MinimalCanonical}
Suppose \(j_0 : V\to M_0\) and \(j_1 : V\to M_1\) are ultrapower embeddings. Then a comparison \((i_0,i_1) : (M_0,M_1)\to P\) of \((j_0,j_1)\) is the canonical comparison of \((j_0,j_1)\) if and only if it is the unique minimal comparison of \((j_0,j_1)\).
\begin{proof}
Suppose \((i_0,i_1) : (M_0,M_1)\to N\) is the canonical comparison of \((j_0,j_1)\). Suppose \((i_0',i_1') : (M_0,M_1)\to N'\) is a minimal comparison of \((j_0,j_1)\). There is some \(h : N \to N'\) such that \(i_0' = h\circ i_0\) and \(i_1' = h\circ i_1\). Since \(i_0'[M_0]\cup i_1'[M_1]\subseteq \text{ran}(h)\), the minimality of \((i_0',i_1')\) implies that \(h\) is surjective. Therefore \(h\) is the identity, and hence \(N = N'\), \(i_0 = i_0'\), and \(i_1 = i_1'\). Thus \((i_0,i_1)\) is the unique minimal comparison of \((j_0,j_1)\).

Conversely suppose \((i_0,i_1) : (M_0,M_1)\to N\) is the unique minimal comparison of \((j_0,j_1)\). Suppose \((i_0',i_1'):(M_0,M_1)\to N'\) is a comparison of \((j_0,j_1)\). Then by \cref{Hull}, there is a minimal comparison \((i_0^*,i_1^*) : (M_0,M_1)\to N^*\) and an elementary embedding \(h: N^*\to N'\) such that \(i_0' = h\circ i_0^*\) and \(i_1' = h\circ i_1^*\). By the uniqueness of \((i_0,i_1)\), \(N^* = N\), \(i_0^* = i_0\) and \(i_1^* = i_1\). Therefore \(h : N \to N'\) witnesses the canonicity of \((i_0,i_1)\) with respect to \((i_0',i_1')\).
\end{proof}
\end{prp}

The following lemma shows that assuming UA one can ``compare comparisons."

\begin{lma}[UA]\label{CompareComparison}
Suppose \(j_0: V\to M_0\) and \(j_1 : V\to M_1\) are ultrapower embeddings. Suppose \((i_0,i_1) : (M_0,M_1)\to N\) and \((i_0',i_1') : (M_0,M_1)\to N'\) are comparisons of \((j_0,j_1)\). Then there is a pair of internal ultrapower embeddings \((k,k') : (N,N')\to P\) such that \(k\circ i_0 = k'\circ i_0'\) and \(k\circ i_1 = k'\circ i_1'\).
\begin{proof}
Let \(\ell = i_0\circ j_0 = i_1\circ j_1\) and let \(\ell' = i_0'\circ j_0 = i_1'\circ j_1\). Thus \(\ell : V\to N\) and \(\ell' : V\to N'\) are ultrapower embeddings. Let \((k,k') : (N,N')\to P\) be any comparison of \((\ell,\ell')\). We claim \((k,k')\) is as desired.

We show \(k\circ i_0 = k'\circ i_0'\). First, since \(k\circ \ell = k'\circ \ell'\) by the definition of a comparison, we have that \(k\circ i_0\restriction j_0[V] = k'\circ i_0'\restriction j_0[V]\). Moreover since \(k\circ i_0\) and \(k'\circ i_0'\) are elementary embeddings from \(M_0\) into the same target model \(P\), by \cref{DefinableUnique}, \(k\circ i_0\restriction \text{Ord} = k'\circ i'_0\restriction \text{Ord}\). Since \(M_0 = H^{M_0}(j_0[V]\cup \text{Ord})\), it follows that \(k\circ i_0 = k'\circ i_0'\) on all of \(M_0\).

A similar argument shows that \(k\circ i_1 = k'\circ i_1'\), and this completes the proof.
\end{proof}
\end{lma}

\begin{thm}[UA]\label{CanonicalExistence}
Every pair of ultrapower embeddings admits a canonical comparison.
\begin{proof}
Fix ultrapower embeddings \(j_0 : V\to M_0\) and \(j_1 : V \to M_1\). It suffices to show that \((j_0,j_1)\) has a unique minimal comparison. Suppose \((i_0,i_1) : (M_0,M_1)\to N\) and \((i_0',i_1') : (M_0,M_1)\to N'\) are minimal comparisons of \((j_0,j_1)\). By \cref{CompareComparison}, there is a pair of internal ultrapower embeddings \((k,k'): (N,N')\to P\) such that \(k\circ i_0 = k'\circ i_0'\) and \(k\circ i_1 = k'\circ i_1'\). By the uniqueness clause of \cref{Hull}, it follows that \((i_0,i_1) = (i_0',i_1')\) as desired.
\end{proof}
\end{thm}

Under UA, canonical comparisons automatically have a stronger universal property:

\begin{defn}
Suppose \(j_0 : V\to M_0\) and \(j_1 : V\to M_1\) are ultrapower embeddings. A comparison \((i_0,i_1) : (M_0,M_1)\to N\) of \((j_0,j_1)\) is a {\it pushout} if for any comparison \((i_0',i_1') : (M_0,M_1)\to N'\) there is a unique internal ultrapower embedding \(h : N\to N'\) such that \(i_0' = h\circ i_0\) and \(i_1' = h\circ i_1\).
\end{defn}

We use the word  ``pushout" since if \((i_0,i_1) : (M_0,M_1)\to N\) is a pushout comparison of \((j_0,j_1)\), then \(N\) is the pushout of \((j_0,j_1)\) in the category of wellfounded ultrapowers of \(V\) with internal ultrapower embeddings. We will prove from UA that canonical comparisons are pushouts. 

We first need another lemma which we often use in conjunction with \cref{InternalClose}:

\begin{lma}\label{FactorUltra}
Suppose \(i' : M\to N'\) is an ultrapower embedding, \(i: M\to N\) is an elementary embedding, and \(h : N\to N'\) is an elementary embedding such that \(i' = h\circ i\). Then \(h\) is an ultrapower embedding.
\begin{proof}
Fix \(a\in N'\) such that \(N' = H^{N'}(i'[M]\cup \{a\})\). Since \(i'[M]\subseteq h[N]\), \(N'  = H^{N'}(h[N]\cup \{a\})\). Therefore \(h\) is an ultrapower embedding.
\end{proof}
\end{lma}

\begin{lma}[UA]\label{Pushout}
Canonical comparisons are pushouts.
\begin{proof}
Suppose \(j_0:V\to M_0\) and \(j_1 : V\to M_1\) are ultrapower embeddings. Suppose \((i_0,i_1) : (M_0,M_1) \to N\) is the canonical comparison of \((j_0,j_1)\). Suppose \((i_0',i_1') : (M_0,M_1)\to N'\) is a comparison of \((j_0,j_1)\). Let \(h : N \to N'\) be the unique elementary embedding with \(i_0' = h\circ i_0\) and \(i_1'= h\circ i_1\). We must show \(h\) is an internal ultrapower embedding. By \cref{FactorUltra}, since \(i_0' = h \circ i_0\), \(h\) is an ultrapower embedding. By \cref{CompareComparison}, fix internal ultrapower embeddings \((k,k') : (N,N')\to P\) such that \(k\circ i_0 = k'\circ i_0'\) and \(k\circ i_1 = k'\circ i_1'\). Since \(k\circ i_0 = k'\circ h\circ i_0\) and \(k\circ i_1 = k'\circ h \circ i_1\), we must have \(k = k'\circ h\) by \cref{MinimalEmb}. Therefore by \cref{InternalClose}, \(h\) is an internal ultrapower embedding, as desired.
\end{proof}
\end{lma}

Translating this from the language of ultrapower embeddings to the language of ultrafilters gives a result on the structure of the Rudin-Frol\'ik order:

\begin{cor}[UA]\label{SemiLattice}
Any pair of countably complete ultrafilters has a least upper bound in the Rudin-Frol\'ik order.
\begin{proof}
Suppose \(U_0\) and \(U_1\) are countably complete ultrafilters. Let \[(i_0,i_1) : (M_{U_0},M_{U_1})\to N\] be the canonical comparison of \((j_{U_0},j_{U_1})\). Let \(\ell : V \to N\) be the ultrapower embedding \(i_0\circ j_{U_0} = i_1\circ j_{U_1}\). Fix any ultrafilter \(W\) such that \(\ell = j_W\). Then \(U_0\D W\) since \(i_0 : M_{U_0}\to M_W\) is an internal ultrapower embedding with \(j_W = i_0\circ j_{U_0}\). Similarly \(U_1\D W\). Thus \(W\) is an upper bound of \(U_0,U_1\) in the Rudin-Frol\'ik order. 

Suppose \(W'\) is another upper bound of \(U_0,U_1\) in the Rudin-Frol\'ik order. Fix \(i_0': M_{U_0}\to M_{W'}\) and \(i_1' :M_{U_1}\to M_{W'}\) witnessing this, so \(i_0'\circ j_{U_0} = j_{W'} = i_1'\circ j_{U_1}\). Thus \((i_0',i_1'): (M_{U_0},M_{U_1})\to M_{W'})\) is a comparison of \((j_{U_0},j_{U_1})\). By \cref{Pushout}, there is an internal ultrapower embedding \(h : N\to M_{W'}\) such that \(i_0' = h\circ i_0\) and \(i_1' = h\circ i_1\). Since \(h : M_W\to M_{W'}\) satisfies \[h\circ j_W = h\circ i_0\circ j_{U_0} = i_0'\circ j_{U_0} = j_{W'}\] \(h\) witnesses that \(W\D W'\). 
\end{proof}
\end{cor}
\subsection{Uniqueness of ultrapower embeddings}
There is a surprisingly powerful consequence of all this category theory:

\begin{thm}[UA]\label{CanonicalInternal}
Suppose \(j_0:V\to M_0\) and \(j_1 : V\to M_1\) are ultrapower embeddings and \((i_0,i_1): (M_0,M_1)\to N\) is the canonical comparison of \((j_0,j_1)\). Suppose \(h : N\to N'\) is an ultrapower embedding. Then the following are equivalent:
\begin{enumerate}[(1)]
\item \(h\) is an internal ultrapower embedding of \(N\).
\item \(h\) is an amenable class of both \(M_0\) and \(M_1\).
\end{enumerate}
\end{thm}

For the proof we need a trivial lemma about compositions of amenable embeddings.
\begin{lma}\label{AmenableComposition}
Suppose \(M\) and \(N\) are transitive models of ZFC, \(i : M\to N\) is an amenable embedding, and \(h : N\to P\) is an elementary embedding that is an amenable class of \(M\) in the sense that \(h\restriction x\in M\) for all \(x\in N\). Then \(h\circ i\) is an amenable embedding of \(M\).
\begin{proof}
Suppose \(x\in M\). We must show that \((h\circ i)\restriction x\in M\). But \(i\restriction x\in M\) and \(h\restriction i(x)\in M\), so \((h\circ i)\restriction x = (h\restriction i(x))\circ (i\restriction x)\in M\).
\end{proof}
\end{lma}

\begin{proof}[Proof of \cref{CanonicalInternal}]
Clearly (1) implies (2). Conversely, assume (2). Let \(i_0' = h\circ i_0\) and \(i_1' = h\circ i_1\). Then \((i_0',i_1') :(M_0,M_1)\to N'\) is a comparison of \((j_0,j_1)\); the fact that \(i_0'\) and \(i_1'\) are internal ultrapower embeddings follows from \cref{AmenableComposition}. Note that \(h\) is the unique elementary embedding such that \(i_0' = h\circ i_0\) and \(i_1' = h\circ i_1\). Therefore since \((i_0,i_1)\) is a pushout by \cref{Pushout}, \(h\) is an internal ultrapower embedding, as desired.
\end{proof}

As an example of an application of all this diagram chasing, we have the following theorem.

\begin{thm}[UA]\label{DivisionInclusion}
Suppose \(j_0 : V\to M_0\) and \(j_1 : V\to M_1\) are ultrapower embeddings. The following are equivalent:
\begin{enumerate}[(1)]
\item There is an internal ultrapower embedding \(i_0: M_0\to M_1\) such that \(j_1 = i_0\circ j_0\).
\item \(M_1\) is an internal ultrapower of \(M_0\).
\item \(M_1\subseteq M_0\).
\end{enumerate}
\begin{proof}
One shows (1) implies (2) implies (3) implies (1). The only implication that is not obvious is that (3) implies (1). 

Assume (3). Let \((i_0,i_1) : (M_0,M_1)\to N\) be the canonical comparison of \((j_0,j_1)\). Since \(M_1\subseteq M_0\), \(i_1\) is an amenable class of \(M_0\). Of course, \(i_1\) is an amenable class of \(M_1\). Therefore by \cref{CanonicalInternal}, \(i_1\restriction N\) is an internal ultrapower embedding of \(N\). It follows that \(i_1\) is an \(\alpha\)-supercompact embedding for all ordinals \(\alpha\). Thus \(i_1\) is the identity. It follows that \(i_0 : M_0\to M_1\) is an internal ultrapower embedding such that \(j_1 = i_0\circ j_0\). Thus (1) holds.
\end{proof}
\end{thm}

\begin{cor}[UA]
If \(M\) is an ultrapower of \(V\) then there is a unique ultrapower embedding \(j : V\to M\).
\begin{proof}
Suppose \(j_0,j_1 : V\to M\). By \cref{DivisionInclusion}, since \(M\subseteq M\), there is an internal ultrapower embedding \(i : M \to M\) such that \(j_1 = i\circ j_0\). By \cref{RKFact}, \(i\) is the identity, so \(j_0 = j_1\).
\end{proof}
\end{cor}

An immediate question is whether the assumption that \(j\) is an ultrapower embedding is necessary here; that is, if \(M\) is an ultrapower of \(V\), is there a unique {\it elementary} embedding \(j : V\to M\)? The answer to this question is also yes, but the proof is quite a bit harder. This appears in the author's thesis.
\section{Factor ultrafilters and the seed order}
\begin{defn}
Suppose \(U\D W\). Let \(k : M_U\to M_W\) be an internal ultrapower embedding with \(j_W=k\circ j_U\). Then \(W/U\) is the \(M_U\)-ultrafilter derived from \(k\) using \([\text{id}]_W\).
\end{defn}
\(W/U\) is well-defined by \cref{RFUnique}. We are agnostic about what the underlying set of \(W/U\) is, except in the case that \(W\) is a uniform ultrafilter on an ordinal, when we make the convention that {\it the underlying ordinal of \(W/U\) should chosen so that in \(M_U\), \(W/U\) is uniform as well.} That is, \(\textsc{sp}(W/U)\) is the least ordinal \(\delta\) such that \(k(\delta)\geq [\text{id}]_W\) where \(k : M_U\to M_W\) is the unique embedding with \(j_W = k\circ j_U\).

\begin{defn}
Suppose \(\alpha\) is an ordinal. The {\it tail filter} on \(\alpha\) is the filter generated by sets of the form \(\alpha\setminus \beta\) for \(\beta < \alpha\). An ultrafilter on \(\alpha\) is called {\it tail uniform}, or just {\it uniform}, if \(U\) extends the tail filter on \(\alpha\).
\end{defn}

\begin{defn}
The {\it seed order} is defined on countably complete uniform ultrafilters \(U_0\) and \(U_1\) by setting \(U_0\swo U_1\) if there is a comparison \((i_0,i_1) : (M_{U_0},M_{U_1})\to N\) of \((j_{U_0},j_{U_1})\) such that \(i_0([\text{id}]_{U_0}) < i_1([\text{id}]_{U_1})\).
\end{defn}

The following theorem is proved in \cite{UA}:

\begin{thm}[UA]
The seed order wellorders the class of uniform countably complete ultrafilters.\qed
\end{thm}

We need the following variant of \cref{DefinableUnique} which appears as \cite{IR} Theorem 3.11.

\begin{thm}\label{MinDefEmb}
Suppose \(M\) and \(N\) are inner models and \(j_0,j_1 : M \to N\) are elementary embeddings. Assume \(j_0\) is definable from parameters over \(M\). Then \(j_0(\alpha) \leq j_1(\alpha)\) for all ordinals \(\alpha\).\qed
\end{thm}

The proof is similar to the proof of the Dodd-Jensen Lemma (see \cite{Steel}) and Woodin's Uniqueness of Close Embeddings Lemma \cite{Woodin}.

Using \cref{MinDefEmb}, we prove a key lemma that leads to the finiteness properties of the Rudin-Frol\'ik order under UA:

\begin{lma}\label{Pushdown}
Suppose \(U\D W\) are nonprincipal countably complete uniform ultrafilters. Then in \(M_U\), \(W/U\swo j_U(W).\)
\begin{proof}
Let \((i_0,i_1) : (M^{M_U}_{W/U}, M^{M_U}_{j_U(W)})\to N\) be a comparison of \((j_{W/U}^{M_U},j_{j_U(W)}^{M_U})\) in \(M_U\). Note that \(M^{M_U}_{W/U} = M_W\) and \(M^{M_U}_{j_U(W)} = j_U(M_W)\), and \((j_U\restriction M_W) : M_W\to j_U(M_W)\) is an elementary embedding. Therefore \(i_0\) and \(i_1\circ (j_U\restriction M_W)\) are elementary embeddings from \(M_W\) to \(N\). Since \(i_0\) is an internal ultrapower embedding, by \cref{MinDefEmb}, \(i_0(\alpha) \leq i_1\circ j_U(\alpha)\) for all ordinals \(\alpha\). In particular, \(i_0([\text{id}]_W) \leq i_1(j_U([\text{id}]_W)) = i_1([\text{id}]^{M_U}_{j_U(W)})\). Therefore \((i_0,i_1)\) witnesses that \(W/U\wo j_U(W)\) in \(M_U\).

We must now show that in fact \(W/U\swo j_U(W)\) in \(M_U\), for which it is enough to show that \(W/U\neq j_U(W)\). If \(W/U = j_U(W)\), however, then \(M_{j_U(W)}^{M_U} = M_W\) and \[j_U\circ j_W = j_{j_U(W)}^{M_U}\circ j_U = j_{W/U}\circ j_U = j_W\] But then \(j_U : M_W\to M_W\) satisfies the hypotheses of \cref{RKFact}, and hence \(j_U\) is the identity. This contradicts our assumption that \(U\) is nonprincipal.
\end{proof}
\end{lma}

The first part of \cref{Pushdown} (showing \(W/U\wo j_U(W)\) in \(M_U\)) is actually part of a much more general phenomenon called the {\it Reciprocity Lemma} (see \cite{SO} Theorem 5.3 and 5.15). The second part (showing \(W/U\neq j_U(W)\)) is just the ultrapower theoretic proof of the well-known fact that \(U\times W\) is not isomorphic to \(W\) if \(U\) is nonprincipal.
\section{The Ultrafilter Factorization Theorem}
In this section, we prove a basic factorization theorem for countably complete ultrafilters.
\begin{defn}
A nonprincipal countably complete ultrafilter \(U\) is {\it irreducible} if for all \(D\D U\), either \(D\) is principal or \(D\) is isomorphic to \(U\).
\end{defn}
The main theorem of this section is that under UA all countably complete ultrafilters factor into irreducibles:
\begin{thm}[UA; Ultrafilter Factorization Theorem]\label{Factorization}
Suppose \(W\) is a countably complete ultrafilter. Then there is a finite linear iterated ultrapower \[V = M_0\stackrel{U_0}\longrightarrow M_1 \stackrel{U_1}{\longrightarrow} \cdots \stackrel{U_{n-1}}\longrightarrow M_n = M_W\]
such that each \(U_i\) is an irreducible ultrafilter of \(M_i\) and \[j_W = (j_{U_{n-1}})^{M_{n-1}}\circ \cdots \circ (j_{U_1})^{M_1}\circ (j_{U_0})^{M_0}\]
\end{thm}
We will also show:

\begin{thm}[UA]\label{Lattice} The Rudin-Frol\'ik order induces a lattice structure on the isomorphism types of countably complete ultrafilters.\end{thm}

These facts come down to the stronger {\it Local Ascending Chain Condition}:
\begin{prp}[UA]\label{ACC}
Suppose \(W\) is a countably complete ultrafilter. Suppose \(W_0\D W_1 \D W_2\D\cdots\) and for all \(n < \omega\), \(W_n\D W\). Then for all sufficiently large \(n < \omega\), \(W_n\) is isomorphic to \(W_{n+1}\).
\begin{proof}
Assume towards a contradiction that the proposition fails. Without loss of generality we may assume that \(W_n\) is not isomorphic to \(W_{n+1}\) for any \(n < \omega\). It follows that for all \(n < \omega\), \(W_{n+1}/W_n\) is nonprincipal in \(M_{W_n}\).

Let \(M_n = M_{W_n}\) and let \(U_{n} = W_{n+1}/W_n\). Thus we have the iterated ultrapower \begin{equation}\label{Iteration}M_0\stackrel{U_0}{\longrightarrow} M_1 \stackrel{U_1}{\longrightarrow} M_2 \stackrel{U_2}{\longrightarrow} M_3 \stackrel{U_3}{\longrightarrow}\cdots\end{equation}

Let \(Z_n = \nicefrac{W}{W_n}\). Then \[Z_{n+1} = \nicefrac{W}{W_{n+1}} = \left(\frac{\nicefrac{W}{W_n}}{\nicefrac{W_{n+1}}{W_n}}\right)^{M_{W_n}} = (\nicefrac{Z_n}{U_n})^{M_{W_n}}\]
The second equality requires an easy formal justification that we omit. 

Let \(M_\omega\) be the direct limit of the iterated ultrapower \cref{Iteration} and let \(j_{n,\omega} : M_n\to M_\omega\) be the direct limit embedding. Then \(M_\omega\) is wellfounded by a standard theorem of Mitchell, and therefore the seed order of \(M_\omega\) is wellfounded.  Let \(Z_n^* = j_{n,\omega}(Z_n)\). It follows from \cref{Pushdown} that \(Z_{n+1}\swo j_{U_n}^{M_n}(Z_n)\) in \(M_{n+1}\). Therefore \(M_\omega\) satisfies \[Z_{n+1}^* = j_{n+1,\omega}(Z_{n+1}) \swo j_{n+1,\omega}(j_{U_n}^{M_n}(Z_n)) = j_{n,\omega}(Z_n) = Z_n^*\]
Therefore in \(M_\omega\), \(Z_0^*\slwo Z_1^*\slwo Z_2^*\slwo\cdots\), contradicting the wellfoundedness of the seed order of \(M_\omega\).
\end{proof}
\end{prp}

For ease of notation, it is somewhat easier to prove the Ultrafilter Factorization Theorem in a somewhat more abstract setting.

\begin{defn}
A partial order \((P,\leq)\) satisfies the {\it local ascending chain condition} if for any \(p\in P\), for any sequence \(p_0\leq p_1\leq\cdots\) such that \(p_n\leq p\) for all \(n < \omega\), for all sufficiently large \(n < \omega\), \(p_n = p_{n+1}\). 
\end{defn}

\begin{defn}
A partial order \((P,\leq)\) is {\it strongly atomic} if for all \(p,q\in P\) with \(p < q\), there is some \(r\) with \(p < r \leq q\) such that the interval \((p,r)\) is empty.
\end{defn}

\begin{lma}[UA]\label{RFProps}
Let \((P,\leq)\) be the class of isomorphism types of countably complete ultrafilters with the partial order induced by the Rudin-Frol\'ik order. Then
\begin{enumerate}[(1)]
\item \((P,\leq)\) is strongly atomic.
\item Assuming \textnormal{UA,} \((P,\leq)\) satisfies the local ascending chain condition.
\end{enumerate}
\begin{proof}
(1) follows from the wellfoundedness of the Rudin-Frol\'ik order. (2) is immediate from \cref{ACC}.
\end{proof}
\end{lma}

\begin{lma}\label{AbstractFactor}
Suppose \((P,\leq)\) is a strongly atomic partial order satisfying the local ascending chain condition. Then for any \(p,q\in P\) with \(p < q\), there is a sequence \(p = p_0< p_1 <\cdots < p_n = q\) such that for all \(i < n\), the interval \((p_i,p_{i+1})\) is empty.
\begin{proof}
One constructs such a sequence \(\langle p_i : i < n\rangle\) by recursion. Let \(p_0 = p\). Suppose \(p_i\) has been defined and \(p_i < q\). Then since \((P,\leq)\) is strongly atomic, we may choose \(p_{i+1}\in P\) such that \(p_i < p_{i+1}\leq q\) and the interval \((p_i,p_{i+1})\) is empty. If \(p_{i+1} = q\), the process terminates, and letting \(n = i+1\), the sequence \(p = p_0< p_1 <\cdots < p_n = q\) is as desired. Otherwise the process continues. 

Assume towards a contradiction that the process never terminates. Then one obtains \(\langle p_i : i < \omega\rangle\) with \(p_0 < p_1 < p_2< \cdots\) and \(p_i \leq q\) for all \(i < \omega\). This contradicts the fact that \((P,\leq)\) has the local ascending chain condition.
\end{proof}
\end{lma}

As a special case of \cref{AbstractFactor}, we have \cref{Factorization}:

\begin{proof}[Proof of \cref{Factorization}]
This is immediate from \cref{RFProps} and \cref{AbstractFactor}.
\end{proof}

We take a similar abstract approach to proving \cref{Lattice}.
\begin{lma}\label{AbstractLattice}
Suppose \((P,\leq)\) is an upper semilattice satisfying the local ascending chain condition. Then \((P,\leq)\) is a lattice.
\begin{proof}
In fact let \(S\subseteq P\) be any nonempty set. We claim \(S\) has a greatest lower bound in \((P,\leq)\). Let \(B = \{p\in P : \forall q\in S\ p\leq q\}\) be the collection of lower bounds of \(S\). Then for any \(p,p'\in B\), the least upper bound \(p\vee p'\) of \(p\) and \(p'\) belongs to \(B\): for any \(q\in S\), \(q\) is an upper bound of \(p\) and \(p'\), so \(p\vee p' \leq q\); hence \(p\vee p'\) is a lower bound of \(S\). In particular, \((B,\leq)\) is directed. 

Since \(S\) is nonempty, \(B\) is bounded (by any element of \(S\)). Therefore since \((P,\leq)\) has the local ascending chain condition, \((B,\leq)\) has the ascending chain condition. It follows that \((B,\leq)\) has a maximal element. But since \((B,\leq)\) is directed, any maximal element of \((B,\leq)\) is in fact the {\it maximum} element. In other words, \(S\) has a greatest lower bound.
\end{proof}
\end{lma}

\begin{proof}[Proof of \cref{Lattice}]
This is an immediate consequence of \cref{RFProps}, \cref{SemiLattice}, and \cref{AbstractLattice}.
\end{proof}
\section{\(\D\) is locally finite}
The main theorem of this section is the following:
\begin{thm}[UA]\label{RFFinite1}
A countably complete ultrafilter \(U\) has at most finitely many predecessors in the Rudin-Frol\'ik order up to isomorphism.
\end{thm}

For the proof, it is more convenient to work not with countably complete ultrafilters up to isomorphism but instead with ultrapower embeddings. We therefore make the following definition:

\begin{defn}
Suppose \(j: V\to M\) and \(j' : V \to M'\) are ultrapower embeddings. The Rudin-Frol\'ik order \(\D\) is defined by setting \(j\D j'\) if and only if there is an internal ultrapower embedding \(k : M\to M'\) such that \(j' = k\circ j\).
\end{defn}

As in the definition above, we will be a little fast and loose in our dealings with relations on proper classes, but everything we do is formalizable in ZFC. \cref{RFFinite1} is an immediate consequence of the following theorem:

\begin{thm}[UA]\label{RFFinite2}
An ultrapower embedding has at most finitely many Rudin-Frol\'ik predecessors.
\end{thm}

\begin{proof}[Proof of \cref{RFFinite1} given \cref{RFFinite2}]
To see that the set of Rudin-Keisler equivalence classes of Rudin-Frol\'ik predecessors of \(U\) is finite, note that the map \(D\mapsto j_D\) passes to a bijection from Rudin-Keisler equivalence classes of countably complete ultrafilters to ultrapower embeddings that maps the collection of Rudin-Frol\'ik predecessors of \(U\) into the set of Rudin-Frol\'ik predecessors of \(j_U\), which is finite by \cref{RFFinite2}.
\end{proof}

\subsection{Dodd parameters}
For the proof of \cref{RFFinite2}, we need some elementary facts about elementary embeddings.

\begin{defn}
A {\it parameter} is a decreasing sequence of ordinals. The {\it parameter order} is the lexicographic order on parameters.
\end{defn}

We often identify parameters with their ranges, which gives a correspondence between parameters and finite sequences of ordinals. Thus if \(p\) and \(q\) are parameters, we will denote by \(p\cup q\) the unique parameter whose range is the union of the ranges of \(p\) and \(q\).

The following standard fact is often quite useful for coding arguments:
\begin{lma}
The parameter order is a wellorder.\qed
\end{lma}

\begin{defn}
Suppose  \(j : V\to M\) is an ultrapower embedding. The {\it Dodd parameter} of \(j\) is the least parameter \(p\) such that \(M = H^M(j[V]\cup \{p\})\).
\end{defn}

Dodd parameters can be defined in greater generality, but here we will only consider the parameters of ultrapowers of \(V\). A useful fact about Dodd parameters is that their minimality persists to elementary extensions:

\begin{lma}\label{RKLemma}
Suppose \(j : V\to M\) and \(k: M \to N\) are ultrapower embeddings. Let \(\ell = k\circ j\). Let \(q\) be the Dodd parameter of \(j\) and \(p\) be the Dodd parameter of \(\ell\). Then \(k(q) < p\).
\begin{proof}
Suppose towards a contradiction that \(p\leq k(q)\). If \(p = k(q)\), then since \(N = H^N(\ell[V]\cup \{p\})\) and \(\ell[V]\cup \{p\}\subseteq k[M]\), we have \(N = k[M]\). Hence \(k\) is surjective, contradicting that \(k\) is a nontrivial elementary embedding. Thus \(p< k(q)\). Since \(N = H^N(\ell[V]\cup \{p\})\), there is some function \(f\) such that \(k(q) = \ell(f)(p)\). Therefore \(N\) satisfies that there is a parameter \(p' < k(q)\) with the property that \(k(q) = k(j(f))(p')\). By elementarity, \(M\) satisfies that there is a parameter \(p' < q\) with the property that \(q = j(f)(p')\). Fixing such a parameter \(p' < q\), \(q\in H^M(j[V]\cup \{p'\})\). Hence \(M = H^M(j[V]\cup q)\subseteq H^M(j[V]\cup \{p'\})\). Thus \(H^M(j[V]\cup \{p'\}) = M\), and since \(p' < q\), this contradicts the minimality of \(q\).
\end{proof}
\end{lma}

Incidentally, this is part of Solovay's original proof of the wellfoundedness of the Rudin-Keisler order (see \cite{KanamoriUltrafilters}): note that if \(j \sRK \ell\), then the Dodd parameter of \(j\) is strictly below that of \(\ell\).

We define a relativized notion of generator that is often useful in the analysis of elementary embeddings.

\begin{defn}
Suppose \(M\) and \(N\) are transitive models of ZFC, \(j : M\to N\) is a cofinal elementary embedding, and \(x\in N\). An ordinal \(\xi\) of \(N\) is an \(x\)-generator of \(j\) if \(\xi\notin H^N(j[M]\cup \{x\}\cup \xi)\).
\end{defn}

The following lemma yields a recursive definition of the Dodd parameter of an elementary embedding.

\begin{lma}\label{DoddGenerator}
Suppose \(j : V\to M\) is an ultrapower embedding. Let \(p\) be the Dodd parameter of \(j\). For all \(m < \textnormal{lth}(p)\), \(p_m\) is the largest \((p\restriction m)\)-generator of \(j\).
\begin{proof}
Since \(M = H^M(j[V]\cup \{p\}) \subseteq H^M(j[V]\cup \{p\restriction m\}\cup (p_m +1))\), there are no \((p\restriction m)\)-generators of \(j\) above \(p_m\). 

To finish, it suffices to show that \(p_m\) is a \((p\restriction m)\)-generator of \(j\). Suppose towards a contradiction that it is not, and let \(q\subseteq p_m\) be a parameter such that \(p_m\in H^{M}(j[V]\cup \{(p\restriction m)\cup q\})\). Letting \(p' = p\setminus \{p_m\}\cup q\), we have \(p' < p\) but \(p\in H^M(j[V]\cup \{p'\})\), and hence \(M = H^M(j[V]\cup \{p'\})\). This contradicts the minimality of \(p\).
\end{proof}
\end{lma} 

\subsection{Finiteness of \(\D\)}
Given these facts, we turn to the proof of \cref{RFFinite2}.

\begin{defn}
Suppose \(\ell : V \to N\) is an ultrapower embedding and \(p\) is its Dodd parameter. For \(m < \text{lth}(p)\), let \(S_m(\ell)\) denote the set of ultrapower embeddings \(j : V\to M\) such that there is an internal ultrapower embedding \(k :  M \to N\) with \(\ell = k\circ j\) and \(k[ M]\subseteq H^N(\ell [V]\cup \{p\restriction m\}\cup p_m)\).
\end{defn}

By definition, \(S_m(\ell)\) is contained in the set of Rudin-Frol\'ik predecessors of \(\ell\). By \cref{RFUnique}, the internal ultrapower embedding \(k : M \to N\) is in fact uniquely determined by the requirement \(\ell = k\circ j\).

The point of the sets \(S_m(\ell)\) is to decompose the collection of proper divisors of \(\ell\) into finitely many pieces:
\begin{lma}\label{Decomposition}
Suppose \(\ell : V \to N\) is an ultrapower embedding and \(p\) is the Dodd parameter of \(\ell\). For any \(j\sD \ell\), \(j\in S_m(\ell)\) for some \(m < \textnormal{lth}(p)\).
\begin{proof}
Suppose \(j : V\to M\) and \(q\) is the Dodd parameter of \(j\). Let \(k : M\to N\) bean internal ultrapower embedding with \(\ell = k\circ j\). Then \(k(q) < p\) in the parameter order.
\end{proof}
\end{lma}

\begin{lma}\label{SStrict}
Suppose \(\ell : V \to N\) is an ultrapower embedding and \(p\) is the Dodd parameter of \(\ell\). For all \(m < \textnormal{lth}(p)\), \(\ell\notin S_m(\ell)\).
\begin{proof}
Suppose towards a contradiction that \(\ell\in S_m(\ell)\). By \cref{RFUnique} (or by the Kunen inconsistency theorem), the unique internal ultrapower embedding \(k : N\to N\) is the identity. Therefore if \(k[N]\subseteq H^N(\ell[V]\cup \{p\restriction m\})\), then \(N = H^N(\ell[V]\cup \{p\restriction m\}\cup p_m)\). But then \(p_m\in H^N(\ell[V]\cup \{p\restriction m\}\cup p_m)\), contradicting \cref{DoddGenerator}.
\end{proof}
\end{lma}

The following easy lemma is not really necessary for the proof of \cref{RFFinite2}, but it will clarify what is going on in \cref{SChar}.

\begin{lma}\label{SClosure}
Suppose \(\ell : V \to N\) is an ultrapower embedding and \(p\) is the Dodd parameter of \(\ell\). For all \(m < \textnormal{lth}(p)\), \(S_m(\ell)\) is closed downwards under \(\D\).
\begin{proof}
Suppose \(j : V\to M\) and \(j' : V\to M'\) are ultrapower embeddings with \(j'\in S_m(\ell)\) and \(j\D j'\). We must show \(j\in S_m(\ell)\).

Fix \(h : M\to M'\) such that \(j' = h\circ j\) and fix \(k' : M'\to N\) such that \(\ell = k'\circ j'\) and \(k'[M']\subseteq H^N(\ell[V]\cup \{p\restriction m\}\cup p_m)\). Let \(k = k'\circ h\). Then \(k: M\to N\) is an internal ultrapower embedding, \(\ell = k\circ j\), and \(k[M]\subseteq k'[M']\subseteq H^N(\ell[V]\cup \{p\restriction m\}\cup p_m)\). Thus \(k\) witnesses that \(j\in S_m(\ell)\), as desired.
\end{proof}
\end{lma}

The key to the proof is the following fact:

\begin{lma}[UA]\label{SDirected}
Suppose \(\ell : V \to N\) is an ultrapower embedding and \(p\) is the Dodd parameter of \(\ell\).  For all \(m < \textnormal{lth}(p)\), \(S_m(\ell)\) is closed under canonical comparisons:  if \(j_0: V\to  M_0\) and \(j_1: V \to  M_1\) belong to \(S_m(\ell)\) and \((i_0,i_1) : ( M_0, M_1)\to M\) is their canonical comparison, then \(i_0\circ j_0\in S_m(\ell)\).
\begin{proof}
Let \(k_0 : M_0\to N\) be the unique ultrapower embedding such that \(\ell = k_0 \circ j_0\). Let \(k_1 : M_1\to N\) be the unique ultrapower embedding such that \(\ell = k_1 \circ j_1\). Then \((k_0,k_1) : (M_0,M_1)\to N\) is a comparison of \((j_0,j_1)\). Therefore by \cref{Pushout}, there is an internal ultrapower embedding \(k : M\to N\) such that \(k_0 = k\circ i_0\) and \(k_1 = k\circ i_1\). By the minimality of canonical comparisons (\cref{CanonicalExistence}), \(M= H^M(i_0[M_0]\cup i_1[M_1])\). Therefore \(k[ M ] = k[H^M(i_0[M_0]\cup i_1[M_1])] = H^N(k_0[M_0]\cup k_1[M_1])\). Since \(j_0,j_1\in S_m(\ell)\), \(k_0[M_0]\cup k_1[M_1]\subseteq H^N(\ell[V]\cup \{p\restriction m\}\cup p_m)\). Therefore \(H^N(k_0[M_0]\cup k_1[M_1])\subseteq H^N(\ell[V]\cup \{p\restriction m\}\cup p_m)\). In other words, \(k[M]\subseteq H^N(\ell[V]\cup \{p\restriction m\}\cup p_m)\). Since we also have \(k\circ (i_0\circ j_0) = k_0\circ j_0 = \ell\), the embedding \(k : M \to N\) witnesses that \(i_0\circ j_0\in S_m(\ell)\), as desired.
\end{proof}
\end{lma}

\begin{cor}[UA]
Suppose \(\ell : V \to N\) is an ultrapower embedding and \(p\) is the Dodd parameter of \(\ell\). For all \(m < \textnormal{lth}(p)\), \(S_m(\ell)\) has an \(\D\)-maximum element.
\begin{proof}
By the local ascending chain condition, \(S_m(\ell)\) has an \(\D\)-maximal element, but by the closure of \(S_m(\ell)\) under canonical comparisons (\cref{SDirected}), this must be an \(\D\)-maximum element.
\end{proof}
\end{cor}

\begin{cor}[UA]\label{SChar}
Suppose \(\ell : V \to N\) is an ultrapower embedding and \(p\) is the Dodd parameter of \(\ell\). For all \(m < \textnormal{lth}(p)\), there is a unique \(j_m \sD \ell\) such that \(S_m(\ell) = \{j : j \D j_m\}\).
\begin{proof}
Let \(j_m\) be the \(\D\)-maximum element of \(S_m(\ell)\). Then by \cref{SClosure}, \(S_m(\ell) = \{j : j \D j_m\}\). By the definition of \(S_m(\ell)\), \(j_m \D \ell\). Finally by \cref{SStrict}, in fact \(j_m \sD \ell\).
\end{proof}
\end{cor}

We can now prove \cref{RFFinite2}.

\begin{proof}[Proof of \cref{RFFinite2}]
Assume by induction that for all ultrapower embeddings \(i \sD \ell\), \(\{j : j\D i\}\) is finite; we will show \(\{j : j\D \ell\}\) is finite. Let \(p\) be the Dodd parameter of \(\ell\). Then for each \(m < \text{lth}(p)\), \(S_m(\ell)\) is finite, since by \cref{SChar} there is some \(j_m \sD \ell\) such that \(S_m(\ell) = \{j : j \D j_m\}\). Note that \[\{j : j \sD \ell\} = \bigcup_{m < \text{lth}(p)} S_m(\ell)\]
Therefore \(\{j : j \sD \ell\}\) is finite, since it is contained in a finite union of finite sets. Therefore \(\{j : j \D \ell\} = \{\ell \}\cup \{j : j \sD \ell\}\) is also finite.
\end{proof}
\section{The Irreducible Ultrafilter Hypothesis}
In this section, we comment on some deeper results related to the study of the Rudin-Frol\'ik order under UA and state some open problems.

The main development since these results is some progress on the general analysis of irreducible ultrafilters.
\begin{defn}
Suppose \(X\) is a set. The {\it Fr\'echet filter on \(X\)} is the filter \(\mathcal F\) generated by sets \(A\subseteq X\) such that \(|X\setminus A| < |X|\). An ultrafilter \(U\) on \(X\) is {\it Fr\'echet uniform} if \(U\) extends the Fr\'echet filter on \(X\).
\end{defn}

Fr\'echet uniform ultrafilters are often just called uniform ultrafilters, but we prefer to distinguish between the two notions of uniformity. The following is a version of one of the main theorems of \cite{SC}.

\begin{thm}[UA]
Suppose \(\lambda\) is a strong limit singular cardinal or a successor cardinal. Suppose \(U\) is a Fr\'echet uniform irreducible ultrafilter on \(\lambda\). Then \(M_U\) is closed under \(\lambda\)-sequences.\qed
\end{thm}

Combined with \cref{Factorization}, this yields a great deal of structure for arbitrary countably complete ultrafilters. Our main open question is whether a complete analysis of irreducible ultrafilters is possible under UA.
\begin{iuh}
Suppose \(U\) and \(W\) are Fr\'echet uniform irreducible ultrafilters. Then either \(U\mo W\), \(W\mo U\), or \(U\) and \(W\) are isomorphic.
\end{iuh}
\begin{qst}
Does UA imply the Irreducible Ultrafilter Hypothesis?
\end{qst}
Assuming the Irreducible Ultrafilter Hypothesis, one can explicitly calculate the comparisons of any pair of ultrafilters in terms of their factorizations into irreducible ultrafilters, so in some sense the Irreducible Ultrafilter Hypothesis (if true) gives a complete explanation of the Ultrapower Axiom. This picture seems slightly too simple; we conjecture that the Irreducible Ultrafilter Hypothesis is refutable from a large cardinal hypothesis.
\bibliography{Bibliography}{}
\bibliographystyle{unsrt}

\end{document}